\documentclass[final]{amsart}
\usepackage{amssymb,amsmath, color}
\usepackage{comment}
\usepackage{graphicx}
\usepackage{enumitem,algorithm2e,algorithmic}
\usepackage{ulem}

\newtheorem{Theorem}{Theorem}[section]

\theoremstyle{definition}
\newtheorem{remark}[Theorem]{Remark}

\newcommand{\R}{{\mathbb R}}

\DeclareMathOperator{\argmin}{argmin}
\newcommand{\interior}{{\rm int}\kern 0.06em}
\newcommand{\inte}{{\rm int}\kern 0.06em}
\newcommand{\cl}{{\rm cl}\kern 0.06em}
\newcommand{\zer}{{\rm zer}\kern 0.06em}
\newcommand{\gph}{{\rm gph}\kern 0.06em}
\newcommand{\dom}{{\rm dom}\kern 0.06em}
\renewcommand{\epsilon}{\varepsilon}

\newcommand{\prox}{{\rm prox}\kern 0.08em}

\newcommand{\cC}{{\mathcal C}}

\newcommand{\cH}{{\mathcal H}}
\newcommand{\cO}{{\mathcal O}}

\newcommand{\ie}{{\it i.e.}\,\,}

\newcommand{\rinf}{\mathbb R\cup\{+\infty\}}
\newcommand{\Rb}{\mathbb R\cup\{+\infty\}}

\begin{document}

\title[Fast decay of solutions to   third order evolution equations]{Fast convex optimization via a third-order in time evolution equation: TOGES-V an improved version of TOGES}

\author{Hedy Attouch}
\address{IMAG, UMR 5149 CNRS, Universit\'e Montpellier,
34095 Montpellier cedex 5, France}
\email{hedy.attouch@umontpellier.fr}

\author{Zaki Chbani}
\address{Cadi Ayyad university, Faculty of Sciences Semlalia, Mathematics, 40000 Marrakech, Morroco. }
\email{chbaniz@uca.ac.ma}

\author{Hassan Riahi}
\address{Cadi Ayyad university, Faculty of Sciences Semlalia, Mathematics, 40000 Marrakech, Morroco. }
\email{h-riahi@uca.ac.ma}

\date{July 4, 2020}

\begin{abstract} 
In a Hilbert space setting $\cH$, for convex optimization, we analyze the fast convergence properties as $t \rightarrow +\infty$ of the trajectories $t \mapsto u(t) \in \cH$ generated by  a third-order in time evolution system.
The  function $f: \cH \to \R$ to minimize is supposed to be convex, continuously differentiable, with  $\argmin_{\cH} f \neq \emptyset$.
It enters into the dynamic through its gradient.  
Based on this new dynamical system, we  improve the results obtained by 
[Attouch, Chbani,  Riahi: Fast convex optimization
via a third-order in time evolution equation, Optimization 2020]. As a main result,  when the  damping parameter  $\alpha$ satisfies
 $\alpha > 3$, we show that   $f(u(t)) - \inf_{\cH} f =   o\left(  1/t^3 \right)$  as $t \rightarrow +\infty$,
as well as  the convergence of the trajectories.
We complement these results by introducing into the dynamic an Hessian driven damping term, which reduces the oscillations. 
In the case of a strongly convex function $ f $, we show an autonomous evolution system of the third order in time with an exponential rate of convergence.
All these results have natural extensions to the case of a convex lower semicontinuous function $f: \cH \to \Rb$. Just replace $f$ with its Moreau envelope.
\end{abstract}

\vspace{0.3cm}

\vspace{0.3cm}

\bigskip
\maketitle

Throughout the paper  $\cH$ is a real Hilbert space, endowed with the scalar product $\langle\cdot,\cdot\rangle$ and the associated norm $\Vert\cdot\Vert$.
Unless specified, $f: \cH \to \R$ is a $\cC^1$ convex function
with $\argmin_{\cH} f \neq \emptyset$. We take $t_0 >0$ as the origin of time (this is justified by the singularity at the origin of the damping coefficient  $\gamma (t) = \frac{\alpha}{t}$  which is used  in the paper).

\section{Introduction of the third-order dynamic (TOGES-V)}
In view of developing fast optimization methods, our study concerns the study of the convergence properties, as $t \rightarrow +\infty$, of the trajectories generated by the third-order in time evolution system
$$
{\rm (TOGES-V)} \quad \dddot u(t) +\frac{\alpha +7}t\ddot u(t)+  \frac{5(\alpha +1)}{t^2} \dot u(t)+ \nabla f\left(u(t)+\frac14t\dot u(t)\right)=0.
$$
As a main result, we will prove the following fast convergence result: 
assuming that the damping parameter satisfies  $\alpha \geq 3$, then for any solution trajectory of the dynamic above
  $$f(u(t)) - \inf_{\cH} f =  { \mathcal O\left( \frac{ 1}{t^3} \right)} \, \mbox{ as } \, t \rightarrow +\infty.$$
  The above system is closely linked to the system (TOGES) introduced by the authors in \cite{ACR-opt} and that we recall below.
\begin{equation*}
{\rm (TOGES)} \quad \dddot{u}(t) + \frac{3\alpha +5}{2t} \ddot{u}(t) +     \frac{3\alpha -1}{t^2} \dot{u}(t)  + \nabla f(u(t) + t \dot{u}(t)) =0.
\end{equation*}  
In \cite[Theorem 2.1]{ACR-opt} the following result has been proved:\\
 Suppose $\alpha \geq 3$. Then, for any trajectory of (TOGES)
 there exists a constant $C>0$ such that, for all $t\geq t_0$ 
\begin{eqnarray*}
&& f(u(t) +t\dot{u}(t) ) - \inf_{\cH} f \leq   \frac{C}{t^3}; \\
&& f(u(t)) - \inf_{\cH}f  \leq \frac{1}{t}\left( t_0(f(x(t_0)) - \inf_{\cH}f )+ \frac{C}{2 t_0^2} \right).
\end{eqnarray*}  
Note that the convergence rate of the values for $u(t)$ is only of order $1/t$, which is not completely satisfactory from the point of view of fast optimization (the classical first-order in time steepest descent does as well). 
By contrast,  we will obtain that for any trajectory generated by (TOGES-V), 
 $$
f(u(t)) - \inf_{\cH}f  =\cO \left(\frac{1}{t^3}\right) \, \mbox{ as } \, t \rightarrow +\infty.
 $$ 
So, (TOGES-V) is a significant amelioration of (TOGES).   
 As for (TOGES), the introduction of (TOGES-V) relies on time rescaling and change of variable technics.
As a starting point for our study, we consider the second-order dynamic
\begin{equation*}
{\rm (AVD)}_{\alpha} \quad \ddot{x}(t)   + \frac{\alpha}{t} \dot{x}(t) +    \nabla f (x(t))  = 0
\end{equation*}
introduced  by  Su-Boyd-Cand\`es \cite{SBC}, and further studied by  Attouch-Chbani-Peypouquet-Redont \cite{ACPR} and May \cite{May}.
The importance of this dynamic comes from the fact that the accelerated gradient method of Nesterov can be obtained as a temporal discretization by taking  $\alpha=3$.
 As a specific feature, the viscous damping coefficient $\frac{\alpha}{t}$ vanishes (tends to zero) as time $t$ goes to infinity, hence the terminology Asymptotic Vanishing Damping with coefficient $\alpha$, ${\rm (AVD)}_{\alpha}$ for short. Let us briefly recall the convergence properties of this system:
  \begin{itemize}
 \item  
 For $\alpha \geq 3$, each trajectory $x(\cdot)$ of ${\rm \mbox{(AVD)}}_{\alpha}$  satisfies the asymptotic convergence rate of the values $f(x(t)) - \displaystyle{\inf_{\cH}f} =\cO \left(1 /t^2\right)$ as $t\to +\infty$, see \cite{AC1}, \cite{ACPR}, \cite{May}, \cite{SBC}.

 \item For $\alpha >3$, it has been shown in   
\cite{ACPR} that each trajectory converges weakly to a minimizer.  The corresponding algorithmic result has been obtained by Chambolle-Dossal \cite{CD}.
In addition, it is shown in \cite{AP} and \cite{May} that the asymptotic convergence rate of the values is  $o(1/t^2)$.

\end{itemize}
These rates are optimal, that is, they can be reached, or approached arbitrarily close.
For further results concerning the system ${\rm \mbox{(AVD)}}_{\alpha}$ one can consult \cite{AAD,AC1,AC2,AC2R-EECT,ACPR,ACR-subcrit,AP,AD,AD17,May,SBC}.

\noindent Let's make  the time rescaling of ${\rm \mbox{(AVD)}}_{\alpha}$ given by
$$t= \tau (s)= s^{\frac{3}{2}}, \quad v(s)= x(\tau (s)).$$
After elementary computation, we obtain the rescaled dynamic (see \cite{ACR-opt}, \cite[Theorem 8.1]{ACR-rescale} and \cite[Corollary 3.4]{ACR-contin} for further details)
\begin{equation}\label{change var77_b}
 \ddot{v} (s)      + \frac{3\alpha -1}{2s}\dot{v} (s) +  \frac{9}{4}s\nabla f(v(s)) =0.
\end{equation}
Since $\alpha \geq 3$ is equivalent to $\frac{3\alpha -1}{2}\geq 4$,
we obtain   that, for $\alpha \geq 3$, 
for any solution trajectory $v(\cdot)$ of 
\begin{equation}\label{change var77_b2}
 \ddot{v} (t)      + \frac{\alpha +1}{t}\dot{v} (t) + t\nabla f(v(t)) =0,
\end{equation}
we have 
 \begin{equation}\label{change var8}
f(v(t))-\inf_{\mathcal H} f
 =\mathcal O \left(\frac{1}{t^{3}}\right) \, \mbox{ as  } t \to + \infty.
\end{equation}
Let's go further, and make a change of the unknown function $v$.  For all $t\geq t_0>0$,
set 
$$v(t)=\frac1{4t^3}\frac{d}{dt}(t^4u(t))=\frac14t\dot u(t)+u(t).$$
Note that $u$ is uniquely determined by $v$ and the Cauchy data.
Then 
$$
\dot v (t)= \frac14t\ddot u (t) +\frac54\dot u (t) \; \hbox{ and }\; \ddot v (t)= \frac14t\dddot u (t)+\frac32\ddot u (t).
$$
  As a consequence,  \eqref{change var77_b} becomes the following third-order evolution system (so doing, we have replaced $4f$ by $f$, which does not affect the convergence rates): 
\begin{equation}
\tag{TOGES-V} \dddot u(t) +\frac{\alpha +7}t\ddot u(t)+  \frac{5(\alpha +1)}{t^2} \dot u(t)+   \nabla f\left(u(t)+\frac14t\dot u(t)\right)=0.
\end{equation}
While keeping a similar structure, this system differs notably from the system (TOGES) analyzed by the authors in \cite {ACR-opt}, hence the suffix V, as Variant. But, as indicated below, these modifications of the dynamics coefficients have important consequences on the convergence rates.
We take for granted the existence of   global solutions in the classical sense of (TOGES-V), which is a direct consequence of the non-autonomous Cauchy-Lipschitz theorem, see for example \cite[Proposition 6.2.1]{haraux91}.

\section{Convergence results}
Let us state our main convergence result.

\begin{Theorem}\label{thm_principal_b}
Let $u: [t_0, +\infty[ \to \cH$ be a solution trajectory of the evolution system {\rm (TOGES-V)}.

\medskip

a) Suppose that $\alpha \geq 3$. Then, as $t \to + \infty$
\begin{itemize}
\item[$(i)$] \, $f(u(t)  + \frac14t\dot u(t)) - \inf_{\cH} f = \displaystyle{\mathcal O \left( \frac{ 1}{t^3} \right)}; $
\item[$(ii)$] \,  $f(u(t)) - \inf_{\cH} f =   \displaystyle{ \mathcal O\left( \frac{ 1}{t^3} \right)}$.
\end{itemize}

b) Suppose that $\alpha > 3$. Then, as $t \to + \infty$

\begin{itemize}
\item[$(iii)$] \, $f(u(t)  + \frac14t\dot u(t)) - \inf_{\cH} f = \displaystyle{o \left( \frac{ 1}{t^3} \right)}; $
\item[$(iv)$] \,  $f(u(t)) - \inf_{\cH} f =   \displaystyle{ o\left( \frac{ 1}{t^3} \right)}$;
\item[$(v)$] \, the trajectory  converges weakly as $t \to + \infty$, let $u(t) \rightharpoonup u_{\infty}$, and $u_{\infty} \in \argmin_{\cH} f$.
\end{itemize}
\end{Theorem}
\begin{proof}
To obtain $(i)$, just replace $v(t) $ by $u(t)  + \frac14t\dot u(t)$  in \eqref{change var8}.\\
 To prove $(ii)$, let's start from the relation
$\frac{d}{dt}\left(t^4u(t)\right)=4t^3v(t)$. After integration from $t$ to $t+h$ of this relation, we get
\begin{eqnarray*}
u(t+h) &= & \left(\frac{t}{t+h}\right)^4u(t)+ \frac{1}{(t+h)^4}\int_t^{t+h}4\tau^3 v(\tau)d\tau\\
	  &= & \left(\frac{t}{t+h}\right)^4u(t)+ \left(1-\left(\frac{t}{t+h}\right)^4\right)\frac{1}{(t+h)^4-t^4} \int_t^{t+h}4\tau^3 v(\tau)d\tau\\
	  &= & \left(\frac{t}{t+h}\right)^4u(t)+ \left(1-\left(\frac{t}{t+h}\right)^4\right)\frac{1}{(t+h)^4-t^4} \int_{t^4}^{(t+h)^4} v(s^{1/4})ds,
\end{eqnarray*}
where the last equality comes from the change of time variable $s=\tau^4$.
According to the convexity of the   function $F=f - \inf_{\cH}f$, and using the Jensen inequality, we obtain
\begin{eqnarray*}
F(u(t+h)) & \leq &  \left(\frac{t}{t+h}\right)^4F(u(t))+ \left(1-\left(\frac{t}{t+h}\right)^4\right)F\left(\frac{1}{(t+h)^4-t^4} \int_{t^4}^{(t+h)^4} v(s^{1/4})ds\right)\\
	& \leq &  \left(\frac{t}{t+h}\right)^4F(u(t))+ \left(1-\left(\frac{t}{t+h}\right)^4\right)\frac{1}{(t+h)^4-t^4}  \int_{t^4}^{(t+h)^4} F\left( v(s^{1/4})\right) ds .
\end{eqnarray*}
Using again the change of time variable $s=\tau^4$, we get
$$
\int_{t^4}^{(t+h)^4} F\left( v(s^{1/4})\right) ds =  \int_t^{t+h}4\tau^3 F\left(v(\tau)\right)d\tau .
$$
It follows
\begin{eqnarray*}
F(u(t+h))  \leq  \left(\frac{t}{t+h}\right)^4F(u(t))+ \left(1-\left(\frac{t}{t+h}\right)^4\right)\frac{1}{(t+h)^4-t^4}  \int_{t}^{t+h} F\left( v(\tau)\right) 4\tau^3d\tau .
\end{eqnarray*}
Using the assertion $(i)$, we obtain the existence of a constant $C>0$ such that
\begin{eqnarray}\label{basic_c}
F(u(t+h)) &\leq&   \left(\frac{t}{t+h}\right)^4F(u(t))+ \left(1-\left(\frac{t}{t+h}\right)^4\right)\frac{C}{(t+h)^4-t^4}\int_t^{t+h}  d\tau  \nonumber\\
&=& \left(\frac{t}{t+h}\right)^4F(u(t))+\frac{Ch}{(t+h)^4}.
\end{eqnarray}
Therefore,
\begin{eqnarray*}
\frac1h\left( (t+h)^4F(u(t+h)) -    t^4F(u(t))\right) \leq C. 
\end{eqnarray*}
By letting $h \rightarrow 0^+$ in the above inequality, we get
$$
\frac{d}{dt}\left(t^4F(u(t))\right) \leq C = \frac{d}{dt}\left(C t\right),
$$
which implies that $t \mapsto t^4F(u(t))-Ct$ is nonincreasing. Consequently,
\begin{equation}\label{ineq-estim}
f (u(t))- \inf_{\cH}f=F(u(t))\leq \left( t_0^4F(u(t_0))-C t_0\right)\frac1{t^4} + \frac{C}{t^3}  .
\end{equation}
We conclude that $f (u(t))- \inf_{\cH}f= \mathcal O\left( \frac{1}{t^3} \right)$ as $t \to +\infty$.

\medskip
 
$iii)$ Take now $\alpha >3$. We know that $v(t)$ converges weakly to some $v_{\infty} \in \argmin_{\cH} f$. The relation
$$\frac{d}{dt}(t^4u(t))=4t^3 v(t)$$
gives after integration
$$
t^4 u(t) = t_0^4 u(t_0) + \int_{t_0}^t 4\tau^3 v(\tau) d \tau.
$$
Equivalently
$$
u(t) =\frac{ t_0^4 u(t_0)}{t^4}  + \frac{ 1}{t^4}\int_{t_0}^t 4\tau^3 v(\tau)  d \tau.
$$
Note that $\int_{t_0}^t 4\tau^3   d \tau \sim t^4$ as $t\to +\infty$.   As a general rule, convergence implies ergodic convergence. Therefore, $u(t)$ converges weakly to  $u_{\infty}= v_{\infty} \in \argmin_{\cH} f$.\\
Moreover we know that 
\begin{equation}\label{change var88b}
f(v(t))-\inf_{\mathcal H} f
 =o\left(\frac{1}{t^{3}}\right) \, \mbox{ as  } t \to + \infty,
\end{equation}
which gives by a similar argument as above,   as  $t \to + \infty$
\begin{eqnarray*}
&& f\left(u(t)  + \frac14t\dot u(t)\right) - \inf_{\cH} f = \displaystyle{o \left( \frac{ 1}{t^3} \right)} \, \mbox{ and } \, f(u(t))-\inf_{\mathcal H} f
 =o\left(\frac{1}{t^{3}}\right).
\end{eqnarray*}
This completes the proof of Theorem \ref{thm_principal_b}.
\end{proof}
\begin{remark}
Let us analyze in more detail  the difference between the two dynamic systems of the third order (TOGES) and (TOGES-V). The key point is the choice of the parameter $p$ in the relation
$$pt^{p-1}v(t)=\frac{d}{dt}(t^pu(t))$$
which connects the two variables $v$ and $u$.
 In \cite{ACR-opt}, we took $ p = 1 $ and thus get  $v(t)=\frac{d}{dt}(tu(t))$. In this case, the convergence of the values for $u$ is only of order $ \mathcal O (1 / t) $. In contrast in (TOGES-V), our choice  is fixed at $ p = 4 $, which allows us to increase the speed of convergence to $\mathcal  O (1 / t^3) $. 
Thus, in Theorem  \ref{thm_principal_b}, we improve the results of  \cite[Theorem 2.1]{ACR-opt}. Indeed, under similar conditions, the convergence rate  of the values $f (u(t))- \inf_{\cH}f$
 passes from  $ \mathcal O\left( \frac{1}{t} \right)$  to $\mathcal O\left( \frac{1}{t^3} \right)$. 
\end{remark}

\section{The third-order dynamic with the Hessian driven damping}
In this section, $f$ is  a convex function which is twice continuously differentiable.  The Hessian of $f$ at $u$ is denoted by $\nabla^2 f (u)$. It belongs to $\mathcal L (\cH, \cH)$, its action on $\xi \in \cH$ is denoted by $\nabla^2 f (u)(\xi)$.
The introduction of the Hessian driven damping into (TOGES-V)
\begin{equation}
 \dddot u(t) +\frac{\alpha +7}t\ddot u(t)+  \frac{5(\alpha +1)}{t^2} \dot u(t)+   \nabla f\left(u(t)+\frac14t\dot u(t)\right)=0,
\end{equation}
 leads us to consider the following system, called (TOGES-VH)
\begin{equation*}
 \dddot u(t) +\frac{\alpha +7}t\ddot u(t)+  \frac{5(\alpha +1)}{t^2} \dot u(t)  +  \beta  \nabla^2 f\Big(u(t) + \frac14 t \dot{u}(t)\Big)\Big(\frac54 \dot{u}(t) + \frac14 t  \ddot{u}(t) \Big)+   \nabla f \Big(u(t) + \frac14 t \dot{u}(t)\Big) =0.
\end{equation*}
The  nonnegative parameters $\alpha$ and $\beta $ are  respectively attached to the viscous damping and to the geometric damping driven by the Hessian. When $\beta =0$, we recover the system (TOGES-V). With respect to (TOGES-V), the suffix "H" refers to Hessian. 
The Hessian driven damping has proved to be an efficient tool to control and attenuate the oscillations of inertial systems, which is a central issue for optimization purposes,  see \cite{AABR}, \cite{ACFR}, \cite{APR2},  \cite{BotCseLas},  \cite{SDJS}. In our context, we will confirm this fact on numerical examples.
Note that $ \nabla^2 f\Big(u(t) + \frac14 t \dot{u}(t)\Big)\Big(\frac54 \dot{u}(t) + \frac14 t  \ddot{u}(t) \Big)$ is exactly the time derivative of $\nabla f \Big(u(t) + \frac14 t \dot{u}(t)\Big)$. This  plays a key role in the following developments.
The following results could be obtained using arguments similar to those developed in the previous section. We choose to present an autonomous proof based on the decreasing properties of an ad hoc Lyapunov function. In doing so, we get another proof of the previous results (take $ \beta = 0 $), which is of independent interest. In addition, we will get explicit values of the constants that enter the convergence rates.
As a main result, when $\beta >0$, thanks to the geometric damping driven by the Hessian, we will obtain a  result of rapid convergence of the gradients.

\subsection{Convergence via Lyapunov analysis} 
The following convergence analysis is based on the decrease property of the function $t \mapsto E(t)$ defined by: given $z \in \argmin_{\cH} f$
\begin{eqnarray}\label{Lyap_E_h}
E(t)&:=& 4(t^3  -2\beta t^2 ) \left(f\Big(u(t)+ \frac14 t\dot{u}(t)\Big) - \inf_{\cH} f \right) \nonumber\\
&&+ \frac{1}{2}\left\Vert 
t^2 \Big( \ddot{u}(t )+  \beta \nabla f(u(t) + \frac14 t \dot{x}(t))\Big) +    (\alpha +5)t\dot{u}(t )+ 4\alpha (u(t )-z)\right\Vert^2. 
\end{eqnarray}
 It is convenient to work with the following condensed formulation of $E(t)$
\begin{equation}\label{Lyap_h_2}
E(t):= 4t \delta (t)  (f(v(t))-\inf_{\cH} f)+\frac{1}{2}\Vert w(t)\Vert^2 ,
\end{equation}
where
\begin{eqnarray}
&& v(t):=u(t)+ \frac14 t\dot{u}(t); \\
&& w(t):=t^2 \Big( \ddot{u}(t )+  \beta \nabla f(u(t) + \frac14 t \dot{u}(t))\Big) +    (\alpha +5)t\dot{u}(t )+4\alpha (u(t )-z);\\
&& \delta(t)= t^2 \left(  1- \frac{2\beta}{t}\right).
\end{eqnarray}

\begin{Theorem}\label{Lyapunov_h}
Let $u: [t_0, +\infty[ \to \cH$ be a solution trajectory of the  evolution equation {\rm (TOGES-VH)}.

\medskip

$(i)$ \, Suppose that the parameters  satisfy the following condition: 
 $$   \alpha > 3 \quad  \mbox{and} \quad    \beta \geq 0.$$
 Then, for all $t\geq  t_1 :=   \frac{2\beta  (\alpha -2)}{\alpha -3} $
\begin{eqnarray*}
&&  f\left(u(t) + \frac14 t \dot{u}(t)\right)-\inf_{\cH} f \leq \frac{(\alpha -2) E (t_1)}{4 t^3}, \\
&& f (u(t))- \inf_{\cH}f \leq \left( t_1^4F(u(t_1))-(\alpha -2) E (t_1) t_1\right)\frac1{t^4} + \frac{(\alpha -2) E (t_1)}{t^3}  ,
\end{eqnarray*}
  with
\begin{eqnarray*}
E(t_1 )&=&4(t_1^3  -2\beta t_1^2 ) \Big(f(u(t_1)+\frac14 t_1\dot{u}(t_1))-\inf_{\cH} f \Big)\\
&+&\frac{1}{2}\Vert 
t_1^2 \Big( \ddot{x}(t_1)+  \beta \nabla f(u(t_1) + \frac14 t_1 \dot{u}(t_1)) \Big) +    (\alpha +5)t_1\dot{u}(t_1 )+4\alpha (u(t_1 )-z)\Vert^2. 
\end{eqnarray*}

\medskip

Moreover, we have an approximate descent method in the following sense:

\medskip

 The function $t \mapsto f (u(t))+\frac{(\alpha -2) E (t_1)}{3t^3}$ is nonincreasing on $[t_1, +\infty[$.

\bigskip

\noindent $(ii)$  In addition, for $\beta >0$ 
$$\int_{t_1 }^{+\infty} t^4  \|\nabla f(u(t) + t \dot{u}(t)) \|^2  dt \leq  \frac{(\alpha -2)E(t_1)}{\beta}.
$$
\end{Theorem}
\begin{proof}
Let us compute the time derivative of the energy function $E(\cdot)$ which is written as follows:
\begin{equation*}
E(t):= 4t \delta (t)  (f(v(t))-\inf_{\cH} f)+\frac{1}{2}\Vert w(t)\Vert^2 .
\end{equation*}
According to the derivation chain rule, we get
\begin{equation}\label{Lyap_h_3}
\dot{E}(t):= 4t \delta (t)  \langle \nabla f(v(t)),  \dot{v}(t )\rangle + 4(\delta (t) + t\dot{\delta} (t) )(f(v(t))-\inf_{\cH} f) +
\langle w(t),  \dot{w}(t)\rangle .
\end{equation}
Let us compute $\dot{w}(t )$.
\begin{eqnarray*} 
\dot{w}(t)&=&t^{2}\left(\dddot{u}(t)+\beta  \nabla^2 f\Big(u(t) + \frac14 t \dot{u}(t)\Big)\Big(\frac54 \dot{u}(t) + \frac14 t  \ddot{u}(t) \Big)  \right) +2t \left(   \ddot{u}(t )+  \beta \nabla f\Big(u(t) + \frac14 t \dot{u}(t)\Big)  \right)\\
&& + (\alpha +5)\Big(t\ddot{u}(t)+\dot{u}(t)\Big)  +4\alpha \dot{u}(t)\\
&=& \left(t^{2}\dddot{u}(t) + (\alpha+7) t \ddot{u}(t) + 5(\alpha +1)\dot{u}(t)+  t^{2}\beta  \nabla^2 f\Big(u(t) + \frac14 t \dot{u}(t)\Big)\Big(\frac54  \dot{u}(t) + \frac14 t  \ddot{u}(t)\Big) \right)\\
&&+ 2t \beta \nabla f\Big(u(t)  + \frac14 t \dot{u}(t) \Big)\\
 &=& -t^2 \nabla f\Big(u(t) + \frac14 t \dot{u}(t) \Big) + 2t \beta \nabla f\Big(u(t) + \frac14 t \dot{u}(t) \Big)
 = -\delta (t) \nabla f(v(t))
\end{eqnarray*}
where the two last inequalities follow from the equation (TOGES-VH) and the definition of $v(\cdot)$.
Let us give an equivalent formulation of $w(t)$ using $v(t)$. This will be useful for the rest of the proof.
According to the definition of $v(t)$, we have $t\ddot{u}(t )= 4 \dot{v}(t ) -5 \dot{u}(t ) $. This gives
\begin{eqnarray} 
 w(t) &=& t^2 \left( \ddot{u}(t )+  \beta \nabla f\Big(u(t) + \frac14 t \dot{u}(t)\Big) \right) +    (\alpha +5)t\dot{u}(t )+4\alpha (u(t )-z) \nonumber \\
&=& t (4 \dot{v}(t ) -5 \dot{u}(t ) ) + t^2 \beta \nabla f(v(t)) 
+  (\alpha +5)t\dot{u}(t )+4\alpha (u(t )-z ) \nonumber \\
&=& 4t \dot{v}(t) + t^2 \beta \nabla f(v(t))+ \alpha (t\dot{u}(t )+4u(t)) -4 \alpha z \nonumber \\
&=& 4t \dot{v}(t) + t^2 \beta \nabla f(v(t))+ 4\alpha (v(t) -z).  \label{Lyap_w}
\end{eqnarray} 
Therefore
\begin{eqnarray} 
\langle   \dot{w}(t), w(t)\rangle &=& -\delta (t) \langle   \nabla f(v(t)), 4t \dot{v}(t) + t^2 \beta \nabla f(v(t))+ 4\alpha (v(t) -z) \rangle \nonumber \\
&=& -\beta t^2 \delta (t) \|\nabla f(v(t)) \|^2 -4t\delta (t) 
\langle   \nabla f(v(t)), \dot{v}(t )\rangle
-4\alpha \delta(t) \langle   \nabla f(v(t)), v(t )-z\rangle \nonumber \\
&\leq & -\beta t^2 \delta (t) \|\nabla f(v(t)) \|^2 -4t\delta (t) 
\langle   \nabla f(v(t)), \dot{v}(t )\rangle
+4\alpha \delta(t) (f(z) - f(v(t)) \label{Lyap_h_4}
\end{eqnarray} 
where the last inequality follows from the convexity of $f$.
Combining (\ref{Lyap_h_3})  with  (\ref{Lyap_h_4}) we obtain
\begin{eqnarray}\label{Lyap_h_5}
\dot{E}(t)&\leq & 4t \delta (t)  \langle \nabla f(v(t)),  \dot{v}(t )\rangle + 4(\delta (t) + t\dot{\delta} (t) )(f(v(t))-\inf_{\cH} f) \nonumber\\
&-&\beta t^2 \delta (t) \|\nabla f(v(t)) \|^2 -4t\delta (t) 
\langle   \nabla f(v(t)), \dot{v}(t )\rangle
+4\alpha \delta(t) (f(z) - f(v(t))).
\end{eqnarray} 
After simplification we get
\begin{equation}\label{Lyap_h_6}
\dot{E}(t) + \beta t^2 \delta (t) \|\nabla f(v(t)) \|^2 +     
4\Big((\alpha -1)\delta (t)  - t\dot{\delta} (t) \Big)
(f(v(t))-\inf_{\cH} f)\leq 0.
\end{equation}
According to the definition of $\delta(\cdot)$
\begin{eqnarray*}
(\alpha -1)\delta (t)  - t\dot{\delta} (t)&=&(\alpha-1)(t^2 -2\beta t)  - t(2t - 2\beta)\\
&=& (\alpha-3)t^2 -2\beta t (\alpha -2).
\end{eqnarray*} 
Therefore, 
\begin{equation}\label{Lyap_h_7}
\dot{E}(t) + \beta t^2 \delta (t) \|\nabla f(v(t)) \|^2 +     
4\Big((\alpha-3)t^2 -2\beta t (\alpha -2) \Big)
(f(v(t))-\inf_{\cH} f)\leq 0.
\end{equation}
For $\alpha > 3$ and $t \geq t_1=   \frac{2\beta  (\alpha -2)}{\alpha -3}$ (\ie $t$ sufficiently large), we have 
$\dot{E}(t) \leq 0$. 
So, for all $t \geq t_1$ we have $E(t) \leq E(t_1)$, which by definition of $E(\cdot)$ gives
$$
4 t^3 \left(  1- \frac{2\beta}{t}\right)(f(v(t))-\inf_{\cH} f) \leq E (t_1).
$$
Note that for $t \geq t_1$ we have $1- \frac{2\beta}{t} \geq \frac{1}{\alpha -2}$. Therefore, for $t\geq t_1$
$$
f\left(u(t) + \frac14 t \dot{u}(t)\right) -\inf_{\cH} f \leq \frac{(\alpha -2) E (t_1)}{4 t^3}.
$$
Moreover by integrating (\ref{Lyap_h_6}) we obtain
$$
\beta \int_{t_1 }^{+\infty} t^2 \delta (t) \left\|\nabla f\left(u(t) + \frac14 t \dot{u}(t)\right) \right\|^2  dt \leq  E(t_1).
$$
By a calculation similar to the one above, we obtain
\begin{equation}\label{Lyap_h_8}
\int_{t_1 }^{+\infty} t^4 \left\|\nabla f\left(u(t) + \frac14 t \dot{u}(t)\right) \right\|^2  dt \leq \frac{(\alpha -2) E (t_1)}{\beta}.
\end{equation}
Then, to pass from estimates on $ y (t) $ to estimates on $ x (t) $, we proceed as in Theorem \ref{thm_principal_b}. Precisely,
define $C= (\alpha -2) E (t_1)$, and set $F= f- \inf f$.
According to $(i)$, we obtained $ t^3 F(v(t) \leq \frac14 C$.\\
On the basis of Jensen's inequality, we obtained in the proof of Theorem \ref{thm_principal_b} that
\begin{equation}\label{ineq-estim_H}
\frac{d}{dt}\left(t^4F(u(t))\right) \leq 
4 \sup_t  t^3 F(v(t)\leq C= (\alpha -2) E (t_1).
\end{equation}
This implies that $t \mapsto t^4F(u(t))-Ct$ is nonincreasing on $[t_1, +\infty[$. Consequently,
\begin{equation}\label{ineq-estim_H_2}
f (u(t))- \inf_{\cH}f\leq \left( t_1^4F(u(t_1))-C t_1\right)\frac1{t^4} + \frac{C}{t^3}  .
\end{equation}
We conclude that $f (u(t))- \inf_{\cH}f= \mathcal O\left( \frac{1}{t^3} \right)$ as $t \to +\infty$.
Let's return to \eqref{ineq-estim_H}. We have
$$
t^4 \frac{d}{dt}\left(F(u(t))\right) + 4t^3 F(u(t)) \leq C.
$$
Since $F(u(t)) \geq 0$, we get
$
\frac{d}{dt}\left(F(u(t))\right) \leq \frac{C}{t^4}.
$
Equivalently
$
\frac{d}{dt}\left(F(u(t)) + \frac{C}{3t^3}\right) \leq 0.
$
Therefore,
$t \mapsto f (u(t)) +\frac{(\alpha -2) E (t_1)}{3t^3} $ is nonincreasing on $[t_1, +\infty[$.
\end{proof}

\subsection{Case $\alpha =3$}
In this case, equation \eqref{Lyap_h_7} becomes
$
\dot{E}(t) \leq     8\beta t (f(v(t))-\inf_{\cH} f),
$
which, after integration, gives
\begin{equation}\label{Lyap_h_777}
E(t) \leq   E(t_0) +   8\beta \int_{t_0}^t \tau  (f(v(\tau))-\inf_{\cH} f) d\tau.
\end{equation}
Set $F(t)= f(v(t))-\inf_{\cH} f$. According to the definition of $E(t)$, we deduce from \eqref{Lyap_h_777} that
\begin{equation}\label{Lyap_h_7777}
4 t \delta(t) F( t) \leq   E(t_0) +   8\beta \int_{t_0}^t \tau  F(\tau) d\tau.
\end{equation}
Let us apply the Gronwall lemma. After elementary compuation, we get
$F( t) \leq C/t^3$. Thus, the conclusions of Theorem \ref{Lyapunov_h} are still valid in the limiting case $\alpha =3$.

\section{Numerical illustrations}
\noindent Let us compare the systems  (TOGES), (TOGES-V), (TOGES-VH). Take $\alpha=3$ for  (TOGES) and  (TOGES-V), so as to respect the condition $\alpha \geq3$, and take $\beta=1$ in (TOGES-VH). We get
 
 \smallskip
  
(TOGES) \, \, \,${\dddot u+\frac{7}t\ddot u+\frac{8}{t^2}\dot u+\nabla f(u+t\dot u)=0}$, 

\medskip

 (TOGES-V) \,
${\dddot u +\frac{10}t\ddot u+  \frac{20}{t^2} \dot u+\nabla f\left(u+\frac{t}4\dot u\right)=0}$,

\medskip

 (TOGES-VH) 
${\dddot u +\frac{10}t\ddot u+  \frac{20}{t^2} \dot u+\nabla f\left(u+\frac{t}4\dot u\right)+\frac14\nabla^2f\left(u(t)+\frac{t}4\dot u\right)\left(5\dot u+4\ddot u\right)=0}$.

\medskip

\noindent Take $\cH= \R \times \R$, and  consider the three test functions:

\vspace{3mm}

\begin{itemize}

\item $f_1(x_1,x_2)= \frac{1}{2}\left( x_1^2+x_2^2\right)$,

\item   $f_2(x_1,x_2) =\frac12\left(x_1+x_2-1\right)^2$ for $(x_1,x_2)\in \mathbb R \times \mathbb R$  
 
\item  $f_3(x_1,x_2)= \left(x_1+x_2^2\right)-2\ln(x_1+1)(x_2+1)$, for $(x_1,x_2)\in \mathbb R_+ \times \mathbb R_+$;

\end{itemize}

\medskip

\noindent Given the  initial conditions $u(0)=(3,1), \dot{u}(0)=\ddot{u}(0)=(0,0)$, we have represented in  Figure \ref{Fig1} for each function $f_i$ ($i=1,2,3$) the trajectories corresponding to  (TOGES) (blue), (TOGES-V) (green) and (TOGES-VH) (red). Without ambiguity, we omit the index $i$ for $f_i$. 
It appears clearly that, for all the  convex functions proposed, (TOGES-V) is faster than  (TOGES)    in convergence of values. This confirms the best estimate of  values $f(u(t)) - \inf_{\cH} f =   \displaystyle{ \mathcal O\left(  1/t^3 \right)}$ obtained in Theorem \ref{thm_principal_b}.
 In this numerical example, we confirm the interest of strengthening the system (TOGES-V) by adding  the Hessian damping. As it appears in the first column of Figure 1, all the trajectories $u(\cdot)$   associated with (TOGES-VH) are more stable and neutralize the oscillations.
\begin{center}
\begin{figure}
\begin{center}
  \includegraphics[width=15.4cm, height=14cm]{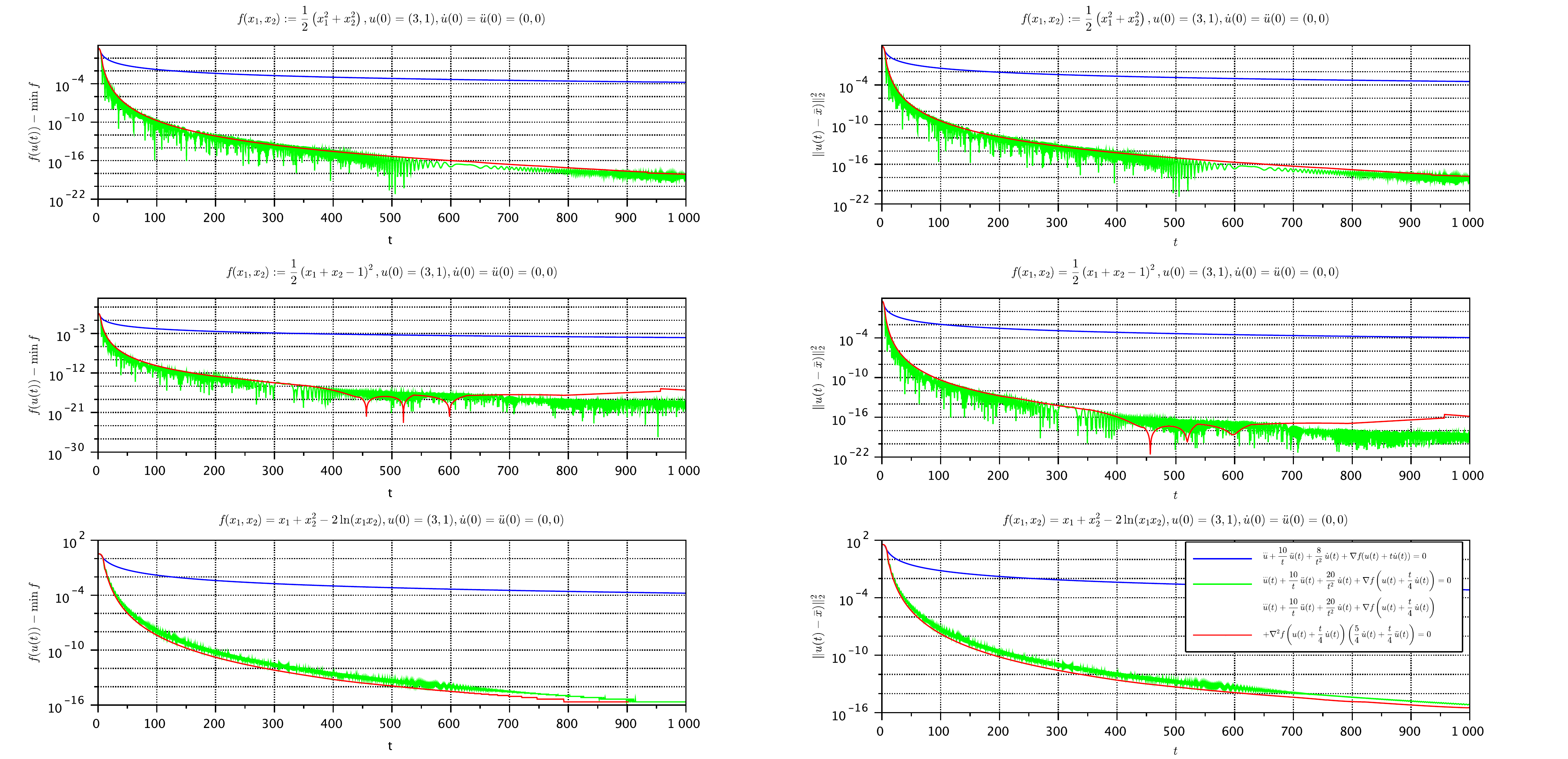}
  \end{center}
  \caption{{\small 
Convergence rates  for (TOGES) (blue),  (TOGES-V) (green), (TOGES-VH) (red)
}} \label{Fig1}
\end{figure}
  \end{center}

\section{The strongly convex case}
\label{sec:sc}
When $f$ is $\mu$-strongly convex, we will show exponential convergence rates for the trajectories generated by the third-order autonomous evolution equation:
\begin{equation}\label{basic_cc}
\dddot{u}(t)+3\sqrt{\mu}\ddot{u}(t)+2\mu\dot{u}(t)+\sqrt{\mu}\nabla f\left(u(t)+\frac1{\sqrt{\mu}}\dot{u}(t)\right)=0.
\end{equation}
\subsection{Lyapunov analysis}
We will develop a Lyapunov analysis based on the decreasing property of the function $t \mapsto \mathcal E (t)$ where, for all $t\geq t_0$
$$
\mathcal E (t):= f\left(  u(t)+\frac1{\sqrt{\mu}}\dot u(t)\right)-  \inf_{\mathcal H}f  + \frac{1}{2} \left\| \sqrt{\mu} ( u(t)+\frac1{\sqrt{\mu}}\dot u(t) -x^*) +  \dot{u}(t) +  \frac1{\sqrt{\mu}}\ddot{u}(t)\right\|^2,
$$
and  $x^*$ is the unique global minimizer of $f$ on $\mathcal{H}$.
To condense the formulas, we write 
$$F(t_0)= f(u(t_0)) - \inf_{\cH}f  \, \mbox{ and } \, C= \mathcal E (t_0) e^{\sqrt{\mu}t_0} .$$
\begin{Theorem}\label{strong-conv-thm}
Suppose that $f: \cH \to \mathbb R$ is $\mu$-strongly convex for some $\mu >0$.
Let  $u(\cdot)$ be a solution trajectory of {\rm(\ref{basic_cc})}.
Then, the following properties hold:  
\medskip
For all $t\geq t_0$, we have 
\begin{eqnarray}
 &&  f\left(u(t)+ \frac1{\sqrt{\mu}}\dot u(t)\right)-  \inf_{\mathcal H}f \leq \mathcal E (t_0) e^{-\sqrt{\mu}(t-t_0)};\label{sc0a}\\
&&  f(u(t)) - \inf_{\cH}f  \leq    \left(C\sqrt{\mu}t + e^{\sqrt\mu t_0}F(t_0) -C \sqrt{\mu}t_0\right)e^{-\sqrt{\mu}t} ;\label{sc0c} \\
 && \| u(t)+ \frac1{\sqrt{\mu}}\dot u(t) -x^*\| ^2 \leq \frac{e^{\sqrt{\mu}t_0}}{\sqrt{\mu}}\left( \sqrt{\mu}  \| y(t_0) -x^*\|^2+2\mathcal E (t_0) (t-t_0 )\right)  e^{-\sqrt{\mu}t};  \label{sc1}\\
 && 
\| u(t) -x^*\|^2  \leq \frac{2}{\mu}  \left(C\sqrt{\mu}t + e^{\sqrt\mu t_0}F(t_0) -C \sqrt{\mu}t_0\right)e^{-\sqrt{\mu}t}.   \label{sc0d}
 \end{eqnarray} 
\end{Theorem}
\begin{proof} (a) 
Set $y(t)= u(t)+ \frac1{\sqrt{\mu}}\dot u(t)$. Note that \eqref{basic_cc} can be equivalently written
\begin{equation}\label{dyn-sc}
\ddot{y}(t) + 2\sqrt{\mu} \dot{y}(t)   + \nabla f (y(t)) = 0,
\end{equation}
and that  $\mathcal E : [t_0, +\infty[ \to \mathbb R^+ $  writes
\begin{equation}\label{E_def_y}
\mathcal E (t) = f(y(t))-  \inf_{\mathcal H}f  + \frac{1}{2} \| \sqrt{\mu} (y(t) -x^*) + \dot{y}(t) \|^2.
\end{equation}
Set $v(t)= \sqrt{\mu} (y(t) -x^*) + \dot{y}(t) $.
Derivation of $\mathcal E (\cdot) $ gives 
$$
\frac{d}{dt}\mathcal E (t)   =  \langle  \nabla f (y(t)),  \dot{y}(t)       \rangle  + \langle v(t), \sqrt{\mu}  \dot{y}(t)  +  \ddot{y}(t)      \rangle .
$$
According to (\ref{dyn-sc}), we obtain
$$
\frac{d}{dt}\mathcal E (t)=\langle  \nabla f (y(t)),  \dot{y}(t)       \rangle  + \langle \sqrt{\mu} (y(t) -x^*) + \dot{y}(t) , -\sqrt{\mu}  \dot{y}(t)  - \nabla f (y(t))       \rangle .
$$
After developing and simplification, we obtain
$$
\frac{d}{dt}\mathcal E (t) + \sqrt{\mu}\langle  \nabla f (y(t)),  y(t) -x^*       \rangle  + \mu\langle y(t) -x^* , \dot{y}(t)\rangle + \sqrt{\mu} \| \dot{y}(t) \|^2   = 0.
$$
According to the strong convexity of $f$, we have
$$
\langle  \nabla f (y(t)),  y(t) -x^*       \rangle  \geq f(y(t))- f(x^*) + \frac{\mu}{2} \| y(t) -x^* \|^2 .
$$
By combining the two relations above, and using the formulation (\ref{E_def_y}) of $\mathcal E (t)$, we obtain 
$$
\frac{d}{dt}\mathcal E (t) + \sqrt{\mu}\left(\mathcal E (t)      
+ \frac{1}{2} \|  \dot{y} \|^2      \right) \leq 0.
$$
Therefore,
$
\frac{d}{dt}\mathcal E (t) + \sqrt{\mu}\mathcal E (t)   \leq 0.
$
By integrating this differential inequality, we obtain, for all $t\geq t_0$
$$
\mathcal E (t)   \leq \mathcal E (t_0) e^{-\sqrt{\mu}(t-t_0)}  .
$$
By definition of $\mathcal E (t)$ and $y(t)$, we deduce that
\begin{equation}\label{basic_sc}
f(y(t))-  \inf_{\mathcal H}f =  f\left(u(t)+ \frac1{\sqrt{\mu}}\dot u(t)\right)-  \inf_{\mathcal H}f 
\leq \mathcal E (t_0) e^{-\sqrt{\mu}(t-t_0)} ,
\end{equation}
and 
\begin{equation*}
\| \sqrt{\mu} (y(t) -x^*) + \dot{y}(t) \|^2=   \left\| \sqrt{\mu} \Big(u(t)+ \frac1{\sqrt{\mu}}\dot u(t) -x^*\Big) + \dot u(t) + \frac1{\sqrt{\mu}}\ddot u(t)\right\|^2 
\leq 2 \mathcal  E (t_0) e^{-\sqrt{\mu}(t-t_0)} .
\end{equation*}
\medskip
 (b) \, Set $C= \mathcal E (t_0) e^{\sqrt{\mu}t_0}$. Developing the expression above, we obtain
\begin{eqnarray*}
\mu \| y(t) -x^*\|^2 + \|  \dot{y}(t)\|^2 
+   \left\langle \dot{y}(t),     2\sqrt{\mu}  ( y(t) -x^*)\right\rangle  \leq 2C e^{-\sqrt{\mu}t}.
\end{eqnarray*}
Note that
\begin{eqnarray*}
\left\langle \dot{y}(t),   2\sqrt{\mu}  ( y(t) -x^*)\right\rangle= \frac{d}{dt} \left(\sqrt{\mu} \| y(t) -x^*\|^2  \right).  
\end{eqnarray*}
Combining the above results, we obtain
\begin{eqnarray*}
\sqrt{\mu}  \left( \sqrt{\mu}  \| y(t) -x^*\|^2 \right) +   \frac{d}{dt} \left( \sqrt{\mu} \| y(t) -x^*\|^2  \right) \leq 2C e^{-\sqrt{\mu}t}.
\end{eqnarray*}
Set $Z(t):= \sqrt{\mu}  \| y(t) -x^*\|^2 $, then we have
$$
\frac{d}{dt} \left( e^{\sqrt{\mu}t}Z(t)-2Ct\right) = e^{\sqrt{\mu}t}\frac{d}{dt} Z(t) + \sqrt{\mu}  e^{\sqrt{\mu}t}Z(t) -2C
 \leq 0.
$$
By integrating this differential inequality,  elementary computation gives 
$$
Z(t)\leq e^{-\sqrt{\mu}(t-t_0)}Z(t_0)+2C(t-t_0 )e^{-\sqrt{\mu}t}.
$$
Therefore
$$
\| y(t) -x^*\| ^2\leq \frac{e^{\sqrt{\mu}t_0}}{\sqrt{\mu}}\left( \sqrt{\mu}  \| y(t_0) -x^*\|^2+2\mathcal E (t_0) (t-t_0 )\right)  e^{-\sqrt{\mu}t}.
$$
 (c) \, Let's now analyze the convergence rate of values
for $ f(u(t))-\displaystyle{\inf_{\mathcal H} f} $. We start from (\ref{basic_sc}), which can be equivalently written  as follows: for all $t\geq t_0$
\begin{equation}\label{Jensen1_sc}
f(y(t))-\inf_{\mathcal H} f
 \leq  C e^{-\sqrt{\mu}t},
\end{equation}
where, we recall, $C= \mathcal E (t_0) e^{\sqrt{\mu}t_0} $.
By integrating the relation $\frac{d}{dt} (e^{\sqrt\mu t} u(t)) =\sqrt\mu e^{\sqrt\mu t}y(t)$
from $t$ to $t+h$ ($h$ is a positive parameter which is intended to go to zero), we get 
$$
e^{\sqrt\mu (t+h)} u(t+h) = e^{\sqrt\mu t} u(t)  + \int_{t}^{t+h}\sqrt\mu e^{\sqrt\mu \tau} y(\tau) d \tau.
$$
Equivalently,
$
u(t+h) =e^{-\sqrt\mu h} u(t)  + e^{-\sqrt\mu (t+h)}\int_{t}^{t+h}\sqrt\mu e^{\sqrt\mu \tau} y(\tau) d \tau.$\\
Let us rewrite this relation in a barycentric form, which is convenient to use a convexity argument:
\begin{eqnarray}
u(t+h) &=&e^{-\sqrt\mu h} u(t)  + e^{-\sqrt\mu (t+h)}\frac{e^{\sqrt\mu (t+h)} -e^{\sqrt\mu t}}{e^{\sqrt\mu (t+h)} -e^{\sqrt\mu t}}\int_{t}^{t+h}\sqrt\mu e^{\sqrt\mu \tau}  y(\tau) d \tau \nonumber \\
 &=& e^{-\sqrt\mu h} u(t)  + (1- e^{-\sqrt\mu h}) \frac{1}{e^{\sqrt\mu (t+h)} -e^{\sqrt\mu t}}\int_{t}^{t+h}\sqrt\mu e^{\sqrt\mu \tau}  y(\tau) d \tau .
\end{eqnarray}
According to the convexity of the   function $f - \displaystyle{ \inf_{\cH}f}$, we obtain
\begin{equation}\label{control_1_sc}
\begin{array}{rcl}
f(u(t+h)) - \inf_{\cH}f & \leq &  e^{-\sqrt\mu h}(f(u(t)) - \inf_{\cH}f ) \\
&+& (1- e^{-\sqrt\mu h})
(f - \inf_{\cH}f) \left(\frac{1}{e^{\sqrt\mu (t+h)} -e^{\sqrt\mu t}}\int_{t}^{t+h}\sqrt\mu e^{\sqrt\mu \tau}  y(\tau) d \tau \right).
\end{array}
\end{equation}
Let us apply the Jensen inequality to majorize the last above expression. Set $s=e^{\sqrt\mu \tau}$, then 
$$
\int_{t}^{t+h}\sqrt\mu e^{\sqrt\mu \tau}  y(\tau) d \tau  = \int_{e^{\sqrt\mu t}}^{e^{\sqrt\mu (t+h)}}  y\left(\frac1{\sqrt\mu }\ln s\right) ds
$$
Let's apply Jensen's inequality. According to the convexity of  $\Psi:=f - \inf_{\cH}f$, we get
\begin{eqnarray*}
\Psi\left(\frac{1}{e^{\sqrt\mu (t+h)} -e^{\sqrt\mu t}}\int_{t}^{t+h}\sqrt\mu e^{\sqrt\mu \tau}  y(\tau) d \tau \right) &=& \Psi\left(\frac{1}{e^{\sqrt\mu (t+h)} -e^{\sqrt\mu t}}\int_{e^{\sqrt\mu t}}^{e^{\sqrt\mu (t+h)}}  y\left(\frac1{\sqrt\mu }\ln s\right) ds\right)\\
 	&\leq & \frac{1}{e^{\sqrt\mu (t+h)} -e^{\sqrt\mu t}}\int_{e^{\sqrt\mu t}}^{e^{\sqrt\mu (t+h)}}  \Psi\left(\ln\left(\frac1{\sqrt\mu }\ln s\right) \right)ds\\
	&=& \frac{1}{e^{\sqrt\mu (t+h)} -e^{\sqrt\mu t}} \int_{t}^{t+h}\sqrt\mu e^{\sqrt\mu \tau} \Psi (y(\tau)) d \tau .
	\end{eqnarray*}
Then, inequality \eqref{control_1_sc} becomes
\begin{equation}\label{control_1_b_sc}
f(u(t+h)) - \inf_{\cH}f \leq  e^{-\sqrt\mu h}(f(u(t)) - \inf_{\cH}f ) + 
\frac{(1- e^{-\sqrt\mu h})\sqrt\mu }{e^{\sqrt\mu (t+h)} -e^{\sqrt\mu t}}\int_{t}^{t+h}e^{\sqrt\mu \tau} (f(y(\tau)) -  \inf_{\cH}f) d \tau .
\end{equation}
According to the convergence rate of the values (\ref{Jensen1_sc}) which was obtained for $y$, we get
\begin{equation}\label{control_1_c_sc}
f(u(t+h)) - \inf_{\cH}f \leq  e^{-\sqrt\mu h}(f(u(t)) - \inf_{\cH}f ) + 
\frac{(1- e^{-\sqrt\mu h})\sqrt\mu }{e^{\sqrt\mu (t+h)} -e^{\sqrt\mu t}}\int_{t}^{t+h}e^{\sqrt\mu \tau} C e^{-\sqrt{\mu}\tau} d \tau .
\end{equation}
Set 
$
F(t) := f(u(t)) - \inf_{\cH}f .
$
We can rewrite equivalently (\ref{control_1_c_sc}) as
$$
F(t+h) - F(t) + (1- e^{-\sqrt\mu h})F(t) \leq Ch \sqrt\mu e^{-\sqrt\mu t}\frac{1- e^{-\sqrt\mu h}}{e^{h}-1}.
$$
Note that $F$ is a $\mathcal C^1$ function, as a composition of such functions.
Therefore, dividing by $h>0$, and letting $h \to 0^+$ in the above inequality, we get by elementary calculation
\begin{equation}\label{control_1_e_sc}
F^\prime(t) + \sqrt{\mu} F(t) \leq C \sqrt{\mu}e^{-\sqrt{\mu}t}.
\end{equation}
Equivalently,
$
\frac{d}{dt}\left( e^{\sqrt\mu t}F(t) -C \sqrt{\mu}t\right)=  e^{\sqrt\mu t} (F^\prime(t) + F(t)) -C \sqrt{\mu}\leq 0.
$
Therefore, the function $t\mapsto e^{\sqrt\mu t}F(t) -C \sqrt{\mu}t$ is nonincreasing, which gives
$$
e^{\sqrt\mu t}F(t) -C \sqrt{\mu}t\leq C_0:=e^{\sqrt\mu t_0}F(t_0) -C \sqrt{\mu}t_0.
$$
We conclude that
$$
f(u(t)) - \inf_{\cH}f  \leq \left(C\sqrt{\mu}t +  C_0\right)e^{-\sqrt{\mu}t} .
$$
(d)   Relations \eqref{sc1}, \eqref{sc0d}
 follow immediately from \eqref{sc0a}, \eqref{sc0c}  and strong convexity of $f$.
\end{proof}
%
\if
{
\begin{remark}{\rm Let's justify the choice of $\gamma = 2\sqrt{\mu}$ in Theorem \ref{strong-conv-thm}.
Indeed,  considering 
$$
\ddot{y}(t) + 2\gamma \dot{y}(t)   
+ \nabla f (y(t)) = 0,
$$
a similar demonstration to that described above can be made on the basis of the Lyapunov function
$$
\mathcal E (t):= f(y(t))-  \inf_{\mathcal H}f  + \frac{1}{2} \| \gamma (y(t) -x^*) + \dot{y}(t) \|^2.
$$
Under the conditions 
$
\gamma \leq \sqrt{\mu} 
$
we obtain the exponential convergence rate
$$
 f(y(t))-  \inf_{\mathcal H}f =\displaystyle{\mathcal O \left( e^{-\gamma t} \right) }\, \mbox{  as } \, t  \,\to + \infty.
$$ 
Taking $\gamma = \sqrt{\mu}$ gives the best convergence rate, and  the result of Theorem \ref{strong-conv-thm}.} 
\end{remark}
}
\fi

\subsection{Numerical illustration}

\noindent The following numerical example illustrates   the linear convergence rate  of  the  third-order evolution equation \eqref{basic_cc} in the case of the strongly convex function $ f(x_1,x_2)=\frac{1}{2}\left(x_1^2+x_2^2\right)$.  
As we can see in the first table of Figure \ref{figure2}, in the case of strongly convex functions, the new system \eqref{basic_cc} (in green color) is more efficient than the two systems (TOGES-V) and (TOGES-VH). However, (TOGES-V) and (TOGES-VH) are quite fast and have the advantage that their rate of convergence is valid for any convex function. 
Linear convergence for the system \eqref{basic_cc}  is illustrated in the second table of Figure \ref{figure2}, where we obtain the estimate  $f(u(t))-\min f  \leq 10^{-20}\exp (-t^{0.47})$  for $100\leq t_0\leq t\leq 10^4$.
The two  systems (TOGES-V) and (TOGES-VH) give the following speed of convergence of values
$f(u(t))-\min f  \leq 10^{4}t^{-6}$.
Finally, note that \eqref{basic_cc} behaves very similar to Polyak's accelerated dynamics for strongly convex functions (in pink color).
\begin{center}
\begin{figure}[h]
\begin{center}
  \includegraphics[width=15.4cm, height=14cm]
  {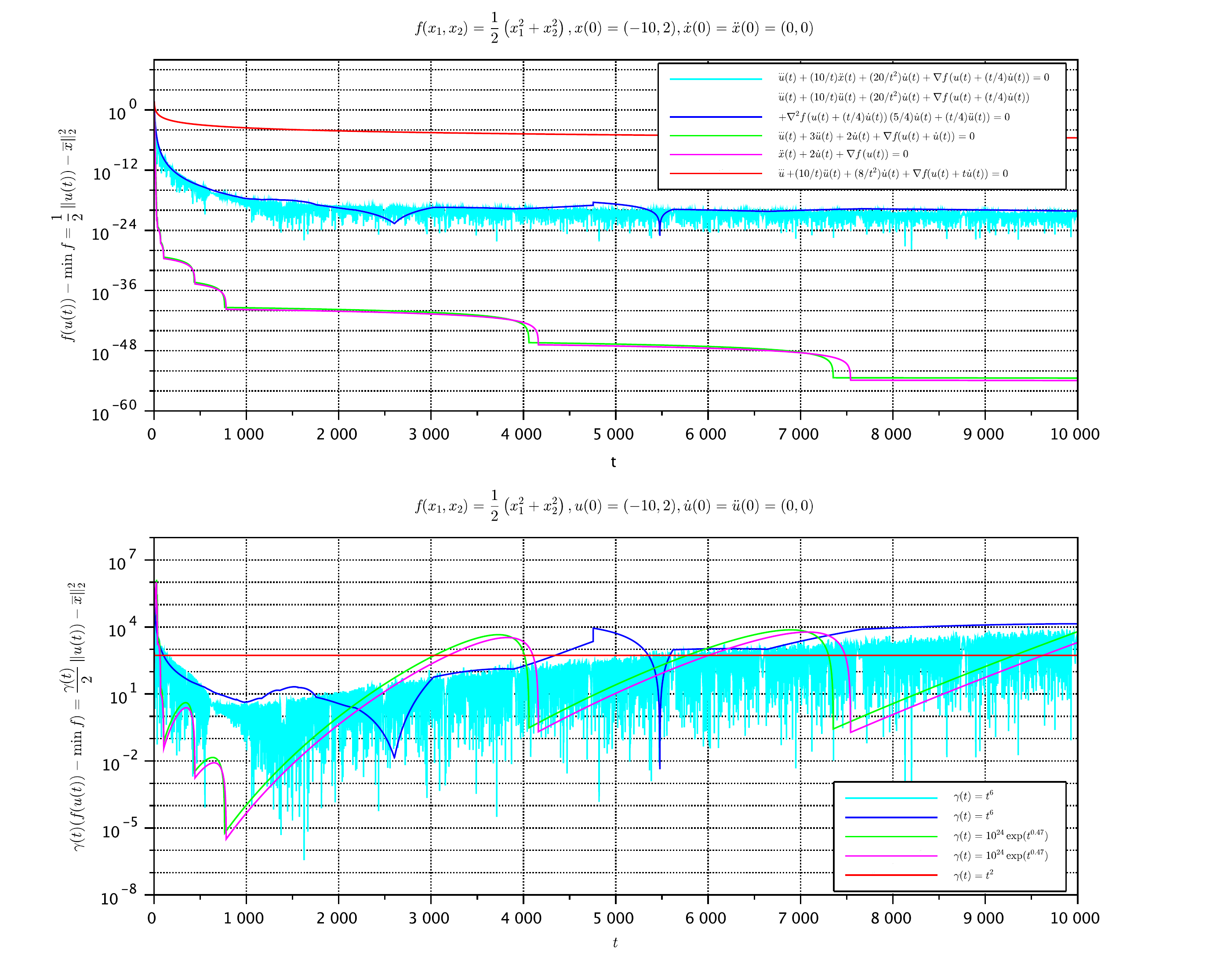}
  \caption{{\small
  Speed of convergence of the values and trajectories associated with different  systems  for the strongly convex function $ f(x_1,x_2)=\frac{1}{2}\left(x_1^2+x_2^2\right)$.}}\label{figure2}
\end{center}
\end{figure}
\end{center}

\section{The nonsmooth case}

In this section, we  assume that $f: \cH \to \rinf$ is a convex lower semicontinuous and proper function such that $\inf_{\cH} f >-\infty$. 
To reduce to the previous situation, where  
$f: \cH \to \R$ is  a $\cC^1$  function whose gradient is Lipschitz continuous, the idea is to replace in (TOGES-V)
$f$ by  its Moreau envelope
 $f_{\lambda}: \cH \to \mathbb R $. Recall that  $f_{\lambda}$ is defined by: for all $x\in \cH$,
$$
f_{\lambda} (x) = \min_{\xi \in \cH} \left\lbrace f (\xi) + \frac{1}{2 \lambda} \| x - \xi \|^2   \right\rbrace.
$$
The function  $f_{\lambda} $ is  convex, of class $ {\mathcal C}^{1,1}$, \, and satisfies  $\inf_{\cH} f_{\lambda} = \inf_{\cH} f $, $\argmin_{\cH} f_{\lambda} = \argmin_{\cH} f$.
We have
\begin{equation}\label{prox_def}
f_{\lambda} (x) = f (\mbox{prox}_{\lambda  f}x) + \frac{1}{2 \lambda} \| x - \mbox{prox}_{\lambda f}x \|^2,
\end{equation}
where $\mbox{prox}_{\lambda  f}$ is the proximal mapping of index $\lambda >0$ of $f$. Equivalently $\mbox{prox}_{\lambda  f} =(I +\lambda \partial f)^{-1}$ is the resolvent of index  $\lambda >0$  of the subdifferential of $f$ (a maximally monotone operator).
One can consult \cite[section 17.2.1]{ABM}, \cite{BaCo}, \cite{Brezis} for an in-depth study of the properties of the Moreau envelope in a Hilbert framework.
 So, replacing $f$ by $f_{\lambda} $ in (TOGES-V), and taking advantage of the fact that $f_{\lambda} $ is continuously differentiable, we get
  \begin{equation*}
\mbox{{\rm(TOGES-VR)} }\qquad \quad \dddot u(t) +\frac{\alpha +7}t\ddot u(t)+  \frac{5(\alpha +1)}{t^2} \dot u(t)+ \nabla f_{\lambda}\left(u(t)+\frac14t\dot u(t)\right)=0,\; t\in [t_0,+\infty[,
\end{equation*}
where the suffix  R stands for Regularized.
Note that $\lambda >0$ is a fixed positive parameter, which can be chosen by the user at his convenience. Indeed, we do not intend to make $\lambda$ tend to zero, as this would lead to a ill-posed limit equation.
 Based on the properties of the Moreau envelope, a direct adaptation of Theorem \ref{thm_principal_b} gives the following convergence results for the regularized dynamic (TOGES-VR).
 
 \begin{Theorem}\label{thm_principal_regularized}
Let $f: \cH \to \rinf$ be  a convex lower semicontinuous and proper function such that $\argmin_{\cH} f \neq \emptyset.$ Let us fix $\lambda >0$.
Let $u: [t_0, +\infty[ \to \cH$ be a solution trajectory of the evolution system {\rm (TOGES-VR)}.

\medskip

a) Suppose that $\alpha \geq 3$. Then, as $t \to + \infty$
\begin{itemize}
\item[$(i)$] \, $f_{\lambda}\Big(u(t)  + \frac14t\dot u(t)\Big) - \inf_{\cH} f = \displaystyle{\mathcal O \left( \frac{ 1}{t^3} \right)}$; \quad
$f\Big( \mbox{\rm prox}_{\lambda  f}\Big(u(t)  + \frac14t\dot u(t)\Big)\Big) - \inf_{\cH} f = \displaystyle{\mathcal O \left( \frac{ 1}{t^3} \right)}$.

\item[$(ii)$] \,  $f_{\lambda}(u(t)) - \inf_{\cH} f =  \displaystyle{ \mathcal O\left( \frac{ 1}{t^3} \right)}$ ; \quad 
$f\Big( \mbox{\rm prox}_{\lambda  f}u(t)\Big) - \inf_{\cH} f = \displaystyle{\mathcal O \left( \frac{ 1}{t^3} \right)}$.
\end{itemize}

b) Suppose that $\alpha > 3$. Then, as $t \to + \infty$

\begin{itemize}
\item[$(iii)$] \,  $f_{\lambda}\Big(u(t)  + \frac14t\dot u(t)\Big) - \inf_{\cH} f = \displaystyle{o \left( \frac{ 1}{t^3} \right)}$; \quad 
$f\Big( \mbox{\rm prox}_{\lambda  f}\Big(u(t)  + \frac14t\dot u(t)\Big)\Big) - \inf_{\cH} f = \displaystyle{o \left( \frac{ 1}{t^3} \right)}$.

\item[$(iv)$] \,  $f_{\lambda}(u(t)) - \inf_{\cH} f =  \displaystyle{ o\left( \frac{ 1}{t^3} \right)}$ ; \quad 
$f\Big( \mbox{\rm prox}_{\lambda  f}u(t)\Big) - \inf_{\cH} f = \displaystyle{o \left( \frac{ 1}{t^3} \right)}$.
\item[$(v)$] \, the trajectory  converges weakly as $t \to + \infty$, let $u(t) \rightharpoonup u_{\infty}$, and $u_{\infty} \in \argmin_{\cH} f$.
\end{itemize}
\end{Theorem}

\begin{proof}
The proof consists in   replacing  $f$ by $f_{\lambda}$ in Theorem \ref{thm_principal_b}, and in using the property
\begin{equation}\label{prox_def_b}
f_{\lambda} (x) \geq f (\mbox{prox}_{\lambda  f}x), 
\end{equation}
which is a direct consequence of the definition \eqref{prox_def} of the proximal mapping.
\end{proof}

\begin{remark}
The temporal discretization of (TOGES-VR), explicit or implicit, naturally leads to relaxed proximal methods, see in the case of second order dynamics \cite{AC_Opt} and \cite{AP-Moreau}.
\end{remark}

\section{Conclusion, perspectives}

From the point of view of the design of first-order rapid optimization algorithms for convex optimization, the new system (TOGES-V) offers 
interesting perspectives.
In its continuous form, this system offers  the convergence rate$f(u(t)) - \inf_{\cH} f =   o\left( \frac{ 1}{t^3} \right)$  as $t \rightarrow +\infty$,
as well as  the convergence of trajectories.
This notably improves the convergence rate $\mathcal O\left( \frac{ 1}{t^2} \right)$ which is attached to the continuous version by Su-Boyd-Cand\`es of the accelerated gradient method of Nesterov.
Since the coefficient of the gradient is fixed in the dynamic, we can expect that the explicit temporal discretization (gradient methods), as well as the implicit temporal discretization (proximal methods), always benefit from similar convergence rates.
This is a topic for further research.
The system (TOGES-V) is flexible, and can be combined with the Hessian damping, which results in a significant reduction in oscillations.
Finally, note that the system (TOGES-V) can be extended to the case of convex lower semicontinuous with extended real values. In this case, the corresponding algorithms are relaxed proximal algorithms, another promising topic.


\end{document}